\documentclass[12pt]{amsart}
\usepackage{amsmath,enumitem}
\usepackage{algorithm}
\usepackage{algpseudocode}
\usepackage{amsfonts}
\usepackage{amssymb}
\usepackage{amsthm}
\usepackage{color}
\usepackage{caption}
\usepackage[dvipsone]{graphicx}
\usepackage[scanall.debug]{psfrag}
\usepackage{epsfig}
\usepackage{geometry}
\usepackage{color}

\usepackage{mathtools}

\newtheorem{theorem}{theorem}[section]
\newtheorem{prop}[theorem]{Proposition}

\newtheorem{lem}[theorem]{Lemma}
\newtheorem{example}[theorem]{Example}
\newtheorem{remark}[theorem]{Remark}
\newtheorem{rem}[theorem]{Remark}

\newtheorem{defin}[theorem]{Definition}
\newtheorem{theor}[theorem]{Theorem}

\newcommand{\downarrowright}[1]{\downarrow
\rlap{\raise0.1cm\hbox{$\scriptstyle{#1}$}}}

\newcommand{\downarrowleft}[1]{\rlap{\kern-0.2cm
\raise0.1cm\hbox{$\scriptstyle{#1}$}}\downarrow}

\newcommand{\uparrowright}[1]{\uparrow
\rlap{\lower0.1cm\hbox{$\scriptstyle{#1}$}}}

\newcommand{\uparrowleft}[1]{\rlap{\kern-0.2cm
\lower0.1cm\hbox{$\scriptstyle{#1}$}}\uparrow}

\begin{document}
\title{Covering-based numbers related to the LS-category of finite spaces\\
}

\author{M. C\'ardenas, R. Flores, A. Quintero and M.T. Villar-Li\~{n}\'an}

\maketitle

\begin{abstract}

In this paper, Lusternik-Schinrelmann and geometric category of finite spaces are considered. We define new numerical invariants of these spaces derived from the geometric category
and present an algorithmic approach for its effective computation. The analysis is undertaken by combining homotopic features of the spaces, algorithms and tools from the theory of graphs and hypergraphs. We also provide a number of examples.

\end{abstract}

\section{Introduction}
\label{intro}

  In \cite{Alex37}, P. Alexandrov observed that finite spaces are the natural topological setting for ordered structures on finite sets. More precisely, the class of finite of posets can be identified with the class of finite $T_0$-spaces.
 \par
 Though these spaces are not relevant from the metric viewpoint, they are far from being irrelevant in Algebraic Topology: in fact, M. McCord proved in \cite{mcCord66} that any compact polyhedron is weakly homotopy equivalent to a finite $T_0$-space. In particular, weak homotopy types of finite $T_0$-spaces coincide with homotopy types of compact polyhedra.
 \par
 After years of oblivion, finite spaces have been recently considered with renewed interest: see \cite{may3} and \cite{Ba11} as comprehensive references and
 \cite{clader09}, \cite{kukiela10} and \cite{raptis10} for more specific aspects of the homotopy theory of such spaces. In particular, the notion of Lusternik-Schnirelmann category (LS-category) in the context of finite spaces was introduced in \cite{vilches15} in connection with the so-called
simplicial LS-category of a simplicial complex, and also in \cite{Tanaka} in connection with the approach to the simplicial complexity in  \cite{Simplicialcomplexity} via finite spaces. The present paper goes further in the study of numerical invariants for the class of finite spaces on its own.

Recall that given a topological space $X$, the LS-category $cat(X)$ of $X$ is defined as the  minimal number of open sets that are contractible in $X$ and cover $X$, while the geometric category $gcat(X)$ is the minimal number of contractible open sets that cover $X$. The latter is not a homotopy invariant, and this leads to the definition of the strong category $Cat(X)$ of $X$ as the smallest value of the geometric category in the homotopy type of $X$. See Section \ref{notation} for more details of these definitions, and Table \ref{tab:coches} for a brief summary of the different definitions of category that appear in the paper.

Besides $cat$ and $Cat$, we introduce other covering-based numbers specifically for the class  of finite $T_0$-spaces. Properties of these numbers, as well as for $cat$ and $Cat$, are given, including the special features of spaces whose Hasse diagrams have height 1. After this analysis, we propose a systematic procedure which allows us to compute or bound some of these numbers using algorithms based on the structure of the space. Attention is paid, in particular, to the complexity of the calculations. It should be pointed out that the lack of an analog to such a procedure in the case of $cat(X)$ makes more difficult the development of an algorithmic approach for this case. See the explanation at the beginning of Section \ref{monocomp}.

Now we describe with detail the contents of the paper. The preliminary Section \ref{notation} contains the necessary definitions and results of the topology of finite spaces and some observations about their LS category. In Section \ref{gcat-cat} we study the behaviour of the function $gcat$ on the homotopy type of a finite space $X$; in particular it is showed that the gap between $cat(X)$ and $gcat(X)$ can be arbitrarily large for finite spaces. Furthermore, the maximum of $gcat$ on the homotopy type of $X$ coincides with its value on the core of $X$. This provides a specific LS-type invariant for finite spaces ($Cat_u$). Prime open sets of a finite space are defined in Section \ref{princat} where it is proved that the geometric category given by them yields a new numerical invariant of LS-type in the class of finite spaces ( $gcat_p$).  Section \ref{height2} is devoted to the finite spaces of height 1, which reveal interesting features: their category is related with the arboricity of the graph given by the Hasse diagram; moreover, all numerical invariants considered in the paper agree on them, and this number is in turn bounded above by the arboricity of a canonically associated multigraph.

The remainder of the paper is devoted to develope a strategy to compute $gcat(X)$ and the other related invariants $Cat_u$ and $gcat_p$ for any finite space $X$. In Section \ref{compalg} we describe a preliminary algorithm that decides if a finite space $X$ is contractible or not, and also identifies the core of $X$. It is checked that the time complexity of the algorithm is at most quartic in the number of points of $X$. Using this algorithm, different procedures (deterministic and heuristic) are designed in Section \ref{identify} in order to describe the compatibility structure of $X$ with respect to $gcat(X)$. As a byproduct, bounds for $gcat(X)$ can be obtained in polynomial time. In Section \ref{monocomp} we generalize the notion of compatibility structure in terms of Boolean functions and define a category associated to such a structure (being natural examples the categories studied in our paper). Moreover, we show that a compatibility structure of a finite space $X$ always gives rise to a hypergraph, in such a way that the associated category corresponds to the covering number of the hypergraph. We conclude by discussing the problem of representing a compatibility structure in a finite space as the compatibility structure of a finite space of height 1.

\textbf{Notation.} We warn the reader that we adopt here the classical definition of LS-category (and its variations); that is, the precise number of open sets involved in the definition instead of the normalized definition which is given by this number minus one. Moreover, the finite spaces that will appear in the text will always be $T_0$, although sometimes this separation condition will not be explicitly quoted.

\textbf{Data availability statement.} Data sharing not applicable to this article as no datasets were generated or analysed during the current study.

\section{Preliminaries}

\label{notation}

Finite spaces are examples of \emph{Alexandrov spaces}; that is, topological spaces whose points admit a minimal open neighbourhood or, equivalently, whose topologies are closed under arbitrary intersections. If $X$ is an Alexandrov space, the minimal open set containing $x\in X$ is denoted $U_x$. Then an ordering can defined on $X$ by setting $x\leq y$ if $U_x\subseteq U_y$. Alexandrov showed in \cite{Alex37} that this ordering yields an equivalence between the class of Alexandrov $T_0$-spaces and the class of posets. Moreover, the homotopy class of an arbitrary Alexandrov space can be represented by an Alexandrov $T_0$-space.
\par
A finite poset $X = (X,\leq)$ is usually represented by its \emph{Hasse diagram}, which turns to be the transitive reduction of $X$. Recall that the  \emph{transitive reduction} of a poset $X$ is the acyclic directed graph whose vertex set is $X$ and a directed edge is drawn from $x$ to $y$ when $x < y$ and there is no $z$ with $x < z < y$ (see \cite{Ba11} or \cite{may3}). In this way, $X$ is then recovered as the transitive closure of the reflexive antisymmetric non-transitive relation defined by the edges of its Hasse diagram, termed the \emph{covering relation} of the poset $X$. Notice that such a diagram, and more generally any acyclic directed graph, admits a decomposition by levels. Namely, minimals of $X$ are placed at level zero and the level assigned to a non-source element $a$ is the number of edges of a maximal directed path from a source to $a$. The \emph{height of $X$} is the maximal height of its elements.
\par
Henceforth we will identify a finite  $T_0$-space $X$  with the Hasse diagram of the corresponding poset without further comment. So, by the height of $X$ we will mean the height of its Hasse diagram.
\par
It is worth pointing out that after removing a point $x\in X$ from a finite space $X$, the resulting Hasse diagram of $X-\{x\}$ is not in general a subgraph of the original one. In fact all edges incident at $x$ in the latter disappear in the former;  and moreover, any pair of edges in the Hasse diagram of $X$ corresponding to $y < x < z$ (if any) is replaced by a directed edge $y < z$ in the Hasse diagram of $X-\{x\}$.
\par
The Hasse diagram of a poset is the $1$-skeleton of the so-called \emph{order complex} of $X$, denoted $\mathcal{O}(X)$. This is the simplicial complex  with the  elements of $X$ as vertices and the totally ordered subsets of $X$ as simplices. Conversely, any simplicial complex $K$ has associated its face poset $F(K)$ consisting of the set of simplices of $K$ ordered by the face relation. McCord theorem in \cite{mcCord66} shows that there exist weak homotopy equivalences $\psi_X: X \to |\mathcal{O}(X)|$ and  $\varphi_K: |K| \to F(K)$ where $|K|$ denotes the underlying polyhedron of the complex $K$. The existence of such weak equivalences show that homotopy types of polyhedra correspond to weak homotopy types of finite $T_0$-spaces.
\par
The natural order on a finite $T_0$-space $Y$ induces an order on the sets of maps $f: X\to Y$ by setting $f\leq g$ if $f(x) \leq  g(x)$ for all $x\in X$. In fact, this order characterizes the homotopies between maps, as proved in  \cite[Corollary 1.2.6]{Ba11}. Namely, two maps $f,g: X\to Y$ are homotopic if and only if there exists a sequence of maps
$f_i: X \to Y$ ($0\leq i\leq m$) such that $f_0 = f$, $f_m = g$ and $f_i$ and $f_{i+1}$ are related; that is, $f_i \leq f_{i+1}$ or $f_i \geq f_{i+1}$. Moreover, by (\cite[Lemma 2.1.1]{Ba11}), we can assume in addition that  there exist points $x_0, \dots, x_m \in X$ such that $f_{i-1} =f_i$ on $X-\{x_{i-1}\}$ and $ f_{i-1}(x_{i-1}) < f_i(x_{i-1})$ or $f_{i-1}(x_{i-1}) > f_i(x_{i-1})$ for $1 \leq i \leq m -1$.
\par
Recall that (co)homology and other algebraic constructions are weak homotopy invariants. In particular, the (co)homology of any compact polyhedron can be realized as the (co)homology of a finite $T_0$-space. However, the LS-category $cat(X)$ of a space $X$ is a homotopy invariant but not a weak homotopy invariant.  Recall that the number $cat(X)$ is defined as the smallest integer $n$ for which there exists an open covering  $\{U_i\}_{i=1}^n$ of $X$ such that for each $i$ the inclusion $U_i\subseteq X$ is homotopically trivial. If such a number does not exist, it is written $cat(X) = \infty$.
\par
\par
Since the LS-category is not a weak homotopy invariant it should not be expected that McCord's theorem yields the equality between $cat(X)$ and $cat(|\mathcal{O}(X)|)$. Indeed, by iterating the face order operator $F(-)$ and the order complex operator $\mathcal{O}(-)$ we get new finite spaces $sd^nX = (F\mathcal{O})^n(X)$, termed the \emph{iterated subdivisions} of $X$ ($n\geq 0$, $sd^0X = X$), such that their corresponding order complexes $\mathcal{O}(sd^nX)$ coincides with the barycentric subdivision $sd^n \mathcal{O}(X)$. Then, as it was observed in \cite{vilches15} the following sequence of inequalities holds:
\begin{equation}\label{desig1}
cat(X)\geq scat(\mathcal{O}(X))\geq cat(sd X)\geq scat(sd \mathcal{O}(X))\geq cat(sd^2 X)\geq \dots \geq cat(|\mathcal{O}(X)|).
\end{equation}
Here $scat(K)$ stands for the \emph{simplicial category} of a simplicial complex $K$; that is, the smallest integer $n$ for which there exists a covering  $\{K_i\}_{i=1}^n$ of $K$ by subcomplexes such that for each $i$ the inclusion $K_i\subseteq K$ is in the contiguity class of a constant map. All these notions in the simplicial setting can be found in \cite{Sco13} and \cite{vilches15}).
\par
It is worth pointing out that given a finite $T_0$-space $X$, we can derive from the sequence of inequalities in (\ref{desig1}) two new invariants of $X$. Namely,

\begin{defin}
The \emph{weak LS-category} of $X$ is the LS-category of the underlying polyhedron of its order complex; that is, the number $cat_w(X) = cat(|\mathcal{O}(X)|)$.
Similarly we can define the \emph{stable LS-category} of $X$ as the number $cat_s(X) = \min \{cat(sd^k X); k\geq 0\}$.
\end{defin}

\begin{prop}
The weak LS-category is a weak homotopy invariant, while the stable LS-category is a homotopy invariant.
\end{prop}

\begin{proof}
If $f: X \to Y$ is a weak homotopy, then the induce simplicial map $\mathcal{O}(f): |\mathcal{O}(X)| \to |\mathcal{O}(Y)|$ is a homotopy equivalence (see \cite[Corollary 1.4.18]{Ba11}) and so $cat_w(X) = cat_w(Y)$. Besides, if $f$ is a homotopy equivalence then  $\mathcal{O}(f)$ is a strong equivalence and then $sd(f) = \mathcal{O}F(f): sd X = \mathcal{O}F(X) \to \mathcal{O}F(Y) = sd Y$ is again a homotopy equivalence by \cite[Theorem 5.2.1]{Ba11}. Therefore, by iterating the argument, we get $cat(sd^k X) = cat(sd^k Y)$ for all $k\geq 0$ and so $cat_s(X) = cat_s(Y)$.
\end{proof}

\begin{rem}\label{Ex1}
Obviously one has the inequalities $cat_w(X) \leq cat_s(X) \leq cat(X)$. Moreover, the gap between $cat(X)$ and $cat_w(X)$ can be arbitrarily large. Indeed, for each $n\geq 2$, let $X$ be the finite space of height 1 whose Hasse diagram is the bipartite graph $K(n,2)$ as shown in the left of Figure \ref{LUZ1wallet}. It is readily checked that $cat(X) = n$ since any open set with two or more maximal points is not contractible in $X$,
 while it is well known that $cat_w(X) = cat(|\mathcal{O}(X)|) = 2$.


In fact the same example shows that the gap between $cat(X)$ and $scat(\mathcal{O}(X))$ (and so $cat_s(X)$) can also be arbitrarily large since $scat(\mathcal{O}(X)) = 2$ according to \cite{vilches15}.
\par
Notice also that the space $X$ in \cite[Example 4.2.1]{Ba11}  (see the at the right in Figure \ref{LUZ1wallet}) is not contractible but $|\mathcal{O}(X)|$ is. Hence $sd^k X$ is not contractible for all $k\geq 0$ by \cite[Corollary 5.2.7]{Ba11}, whence $cat_s(X) \geq 2  >  cat_w(X) = 1$. Furthermore, one easily checks that $cat(X) = 2$ and so $cat(X) = cat_s(X) = 2 > 1 = cat_w(X)$.

\begin{figure}[ht]
	\centering
	\psfrag{a0}{$a_0$}\psfrag{b0}{$b_0$}\psfrag{x1}{$x_1$}\psfrag{x2}{$x_2$}\psfrag{x3}{$x_3$}\psfrag{xn2}{$x_{n-2}$}\psfrag{xn1}{$x_{n-1}$}\psfrag{xn}{$x_{n}$}\psfrag{l}{$\cdots$}\psfrag{d}{$\vdots$}\psfrag{a1}{$a_1$}\psfrag{b1}{$b_1$}\psfrag{ah1}{$a_{h-1}$}\psfrag{bh1}{$b_{h-1}$}\psfrag{ah2}{$a_{h-2}$}\psfrag{bh2}{$b_{h-2}$}\psfrag{ah3}{$a_{h-3}$}\psfrag{bh3}{$b_{h-3}$}	
	\psfig{figure=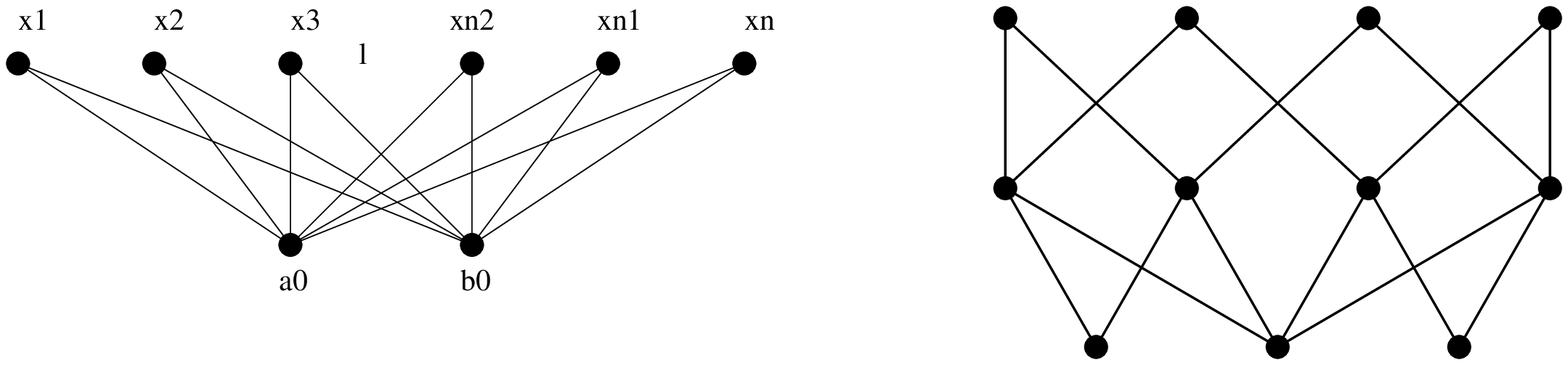,height=2cm}
	\caption{}
	\label{LUZ1wallet}
\end{figure}
\end{rem}
\begin{center}
{\bf Open Question:} Find a finite $T_0$-space $X$ with  $cat(X) > cat_s(X) > cat_w(X)$.
\end{center}
If we consider the LS category of maps, Tanaka shows in \cite{tanaka2} the equality $cat_w(X) = cat_s(id)$, where, for a map between finite spaces $f: X \to Y$, $cat_s(f) = \min\{cat(f\circ \tau^k); k\geq 0\}$ with $\tau^k$  the iterated subdivision map $\tau^k: sd^kX \to sd^{k-1}X \to \dots \to sdX \to X$.

\section{The geometric category and the strong category of finite spaces}\label{gcat-cat}

If we use open sets which are contractible in themselves, we get the \emph{geometric category} $gcat(X)$ of $X$, which is the minimal $n\geq 1$ such that there is a covering of $X$ with $n$ contractible open sets. It turns out that, in general,  the geometric category is not a homotopy invariant. However, it is defined a new homotopy invariant, termed the \emph{strong category} of $X$ and denoted $Cat(X)$,  by considering the minimal value of $gcat(Y)$ for all spaces $Y$ with the homotopy type of $X$.  We refer to \cite{CLOT03} for a comprehensive treatment of these (and others) numerical invariants in Homotopy Theory.
\par
A distinctive property of finite spaces is the fact that the homotopy type of any finite $T_0$-space is represented up to homeomorphism by a minimal space. i.e. a space without beat points. Recall that a point $x\in X$ is termed an \emph{up beat point} if the set of points which are greater than $x$ has a minimum. Similarly, $x$ is said to be a \emph{down beat point} if the set of points below it has a maximum. If we do not distinguish if $x$ is an up or a down beat point, we simply say that $x$ is a beat point. It is immediate that if $x \in  X$ is a beat point, there exists $y \in X$, $y \neq x$, such that any point which is comparable with $x$ is also comparable with $y$. In \cite{Stong66}, R. Stong showed that for any beat point $x \in X$, the inclusion
$X-\{x\}\subseteq X$ is a strong deformation retract and that after removing the beat points, one at a
time, we obtain a strong deformation retract of X with no beat points, called the \emph{core} of $X$, which is unique up to homeomorphism. A $T_0$-space is called \emph{minimal} if it has no beat point, and the homotopy type of a finite $T_0$-space contains a unique minimal space up to homeomorphism (see \cite{Ba11}).

\par
As the LS-category is a homotopy invariant, it will suffice to compute it for minimal spaces. In contrast, as observed in \cite{vilches15}, the removal of beat points may increase the geometric category. In fact, the following proposition shows that this may occur only for up beat points.

\begin{prop}\label{removingdown}
	Let $b$ be a down beat point in a finite $T_0$-space $X$. Then $gcat(X)=gcat(Y)$ for $Y=X-\{b\}$.
\end{prop}

\begin{proof}
	We may assume that $X-Y$ reduces to a down beat  point $b$.
	\par
	Let $\mathcal{V} = \{V_1, \dots, V_n\}$ be an open covering of $Y$ consisting of contractible sets. Since $b$ is a down beat point, let  $c$ be the maximum of the points below $b$ in $X$. Let $m\in Max(X)$ with $b\leq m$. If $b = m$, then choose $V_i$ with $c\in V_i$. The open set $V'_i = V_i \cup \{b\}$ is contractible in $X$ ($b$ is a down beat point in $V'_i$) and $\{V_1, \dots, V'_i, \dots V_n\}$ covers $X$. Otherwise, if $b < m$, take $V_i$ with $m\in V_i$ and so $U^Y_m \subseteq V_i$, where $U^Y_m$ is the minimal open set of $m$ in $Y$. Then $U^X_m = U^Y_m \cup\{b\}$ and $V'_i = V_i\cup \{b\}$ is an open set in $X$. Moreover, $b$ remains a down beat point in $V'_i$, and so by replacing $V_i$ by $V'_i$ in $\mathcal{V}$ we get an open covering of $X$ consisting of contractible sets. This shows that $gcat(X)\leq gcat(Y)$.
	
	On the other hand, it is easy to see that given any covering $\mathcal{W}=\{W_1, \dots, W_n\}$ of $X$ by contractible open sets one can obtain a covering of $Y$ by contractible open sets of the same cardinality. Indeed, if $b\in W_i$ then $b$ is also a down beat point of $W_i$ and hence $W_i'=W_i-\{b\}$ is an open contractible open set in $Y$. Hence $gcat(X)=gcat(Y)$.
\end{proof}
In contrast, as mentioned above, removing up beat points may increase $gcat$. See Example \ref{Expalo}(2) below.

\begin{example}\label{Expalo}
\begin{enumerate}
\item The following example shows a finite $T_0$-space such that $gcat(X) = 3$ but $cat(X) = 2$ since the union $U_a\cup U_c$, although it is not contractible in itself, it is contractible in $X$.

\begin{figure}[ht]
	\centering
	\psfrag{a}{$b$}\psfrag{b1}{$a$}\psfrag{b2}{$c$}\psfrag{c1}{$c_1$}\psfrag{c2}{$c_2$}\psfrag{d1}{$d_1$}\psfrag{d2}{$d_2$}
	\psfrag{e1}{$e_1$}\psfrag{e2}{$e_2$}
	\psfig{figure=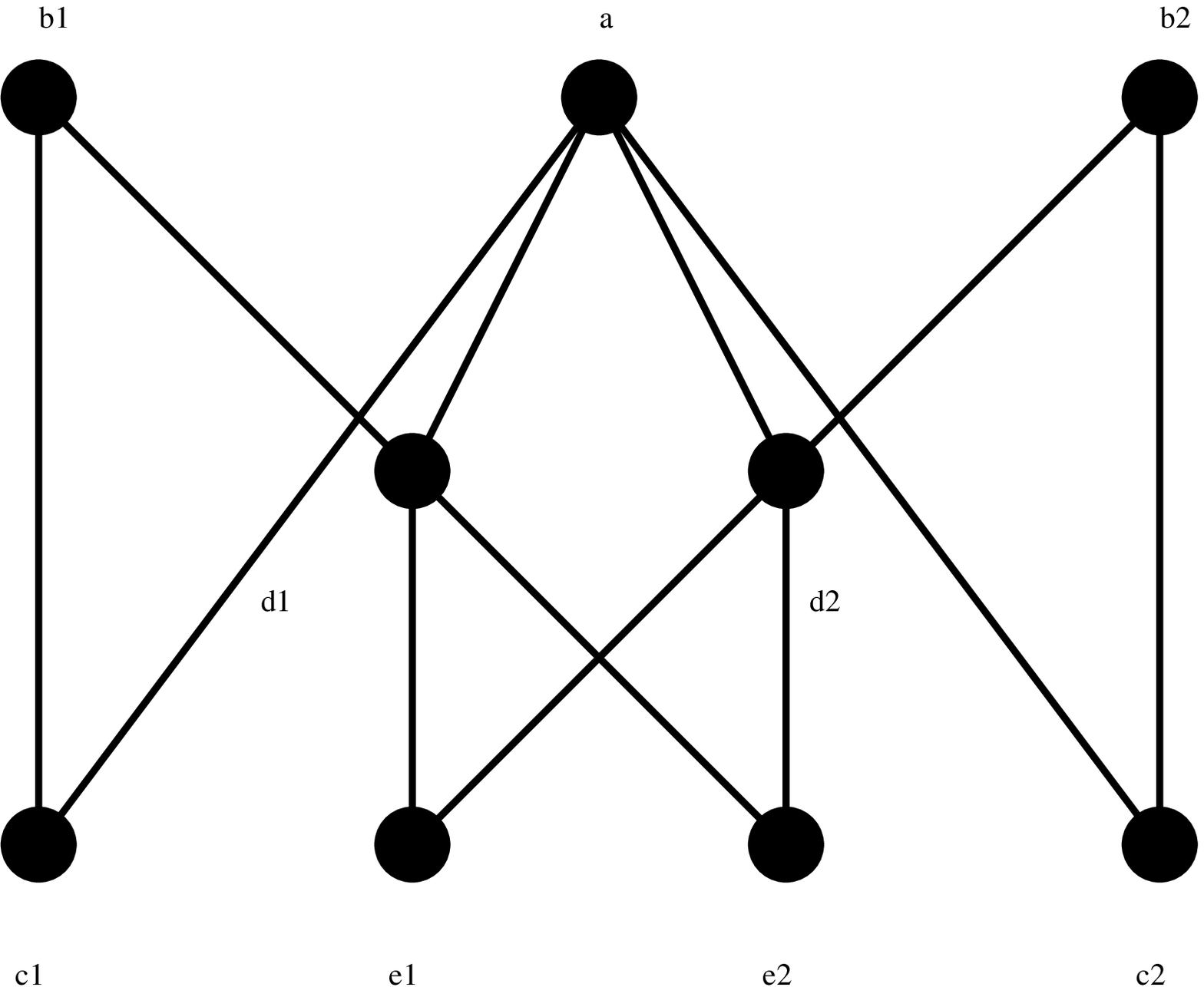,height=2.5cm}
	\caption{}
	\label{counter}
\end{figure}

\item The gap between $gcat$ and $cat$ can be arbitrarily large. Indeed, given any integer $s\geq 1$, for the space $X(s)$ depicted in Figure \ref{dibgcat0},  whose core is showed in Figure \ref{dibgcat0},  one checks $cat(X(s))=gcat(X(s))=2$ since the contractible open sets $U_b$ and $U_e\cup U_{j_0}\cup\cdots\cup U_{j_s}$ cover $X(s)$. However by deleting the up beat point $e$ from $X(s)$, for $Y(s)=X(s)-\{e\}$ in Figure \ref{dibgcat0} it is verified that $cat(Y(s))= cat(X(s) =2$ while $gcat(Y(s))=s+2$ since the open sets $U_{j_t}$ are disjoint from each other and $U_b\cup U_{j_t}$ are not contractible, for $t=0,\dots , s$. Notice that $gcat(Y(s)) = gcat(Z(s))$ since the removing any of the up beat points $c_0,\dots ,c_s$ and $l_0, \dots, l_s$ does not change $gcat$.
\begin{center}
	\begin{figure}[ht]
		\centering
		\psfrag{b}{$b$}\psfrag{b0}{$b_0$}\psfrag{bs}{$b_s$}\psfrag{c}{$c$}\psfrag{c0}{$c_0$}\psfrag{cs}{$c_s$}\psfrag{d}{$d$}\psfrag{d0}{$d_0$}\psfrag{ds}{$d_s$}\psfrag{e}{$e$}\psfrag{e0}{$e_0$}\psfrag{es}{$e_s$}\psfrag{j0}{$j_0$}\psfrag{js}{$j_s$}\psfrag{l0}{$l_0$}\psfrag{ls}{$l_s$}\psfrag{m0}{$m_0$}\psfrag{ms}{$m_s$}\psfrag{x}{$X(s)\equiv$}\psfrag{y}{$Y(s)\equiv$}
		\psfrag{z}{$Z(s)\equiv$}
		\psfig{figure=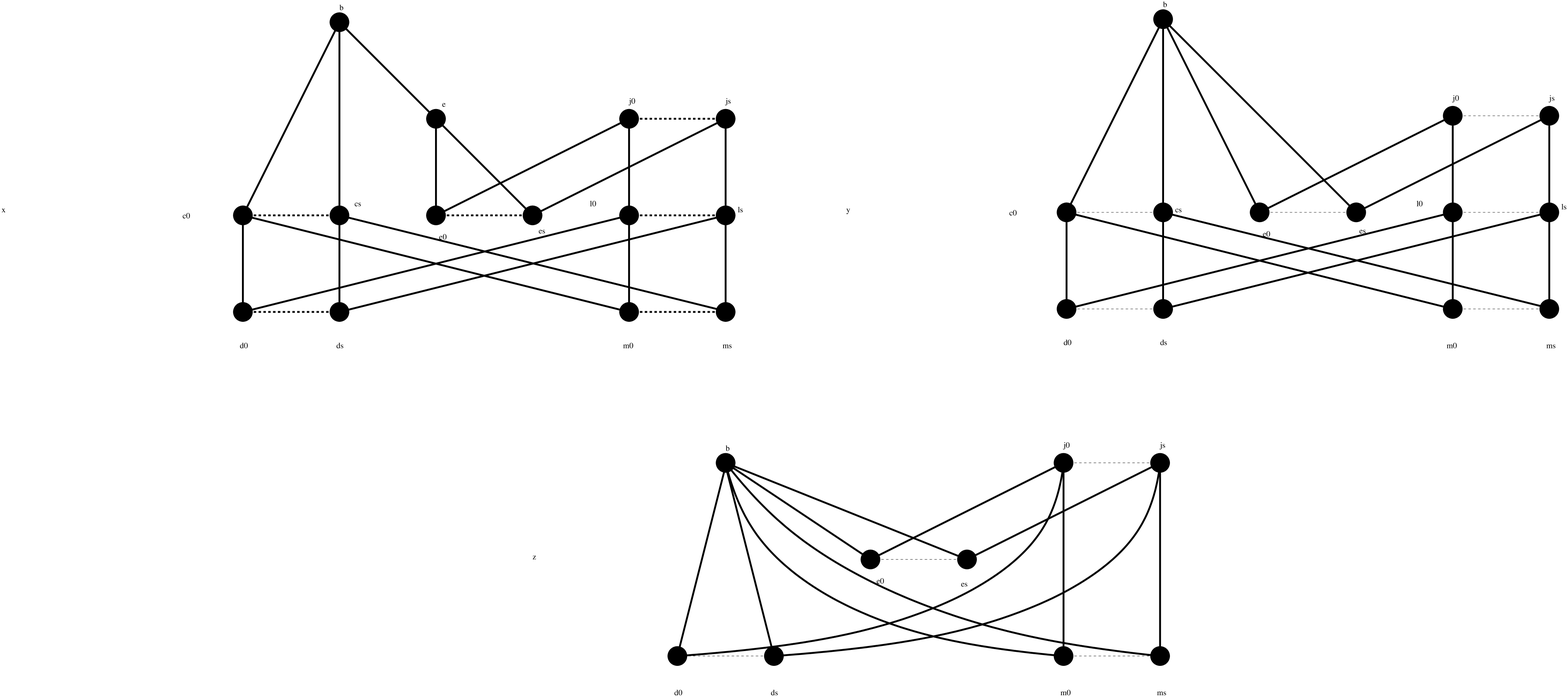,height=6cm}
		\caption{}
		\label{dibgcat0}
	\end{figure}
\end{center}

\end{enumerate}
\end{example}

Recall that, by definition, $Cat(X)$ is the minimum of the function $gcat$ on the homotopy type of $X$. Moreover, as observed above the removal of beat points never diminishes the geometric category, and so the maximum of $gcat$ on the homotopy type of $X$ is precisely $gcat(X_0)$ where $X_0 \subseteq X$ is the core of $X$. Moreover, as $X_0$ is unique up to homeomorphism within the homotopy class of $X$, we have a new numerical homotopy invariant for finite $T_0$-spaces. Namely,

\begin{defin}
Given a finite $T_0$-space $X$, we define the \emph{upper strong category} of $X$, $Cat_u(X)$, to be geometric category of its core $X_0 \subseteq X$.
\end{defin}

\begin{example}\label{CatyCatu}
Notice that the gap between $Cat(X)$ and $Cat_u(X)$ can also be arbitrarily large as the following example shows.

\begin{center}\label{dibgcat8b}
\psfrag{x}{$X\equiv$}\psfrag{p0}{$p_0$}\psfrag{p1}{$p_1$}\psfrag{p2}{$p_2$}\psfrag{pn}{$p_n$}\psfrag{q1}{$q_1$}\psfrag{q2}{$q_2$}\psfrag{q3}{$q_3$}\psfrag{q4}{$q_4$}\psfrag{q2n1}{$q_{2n-1}$}\psfrag{q2n}{$q_{2n}$}
\psfig{figure=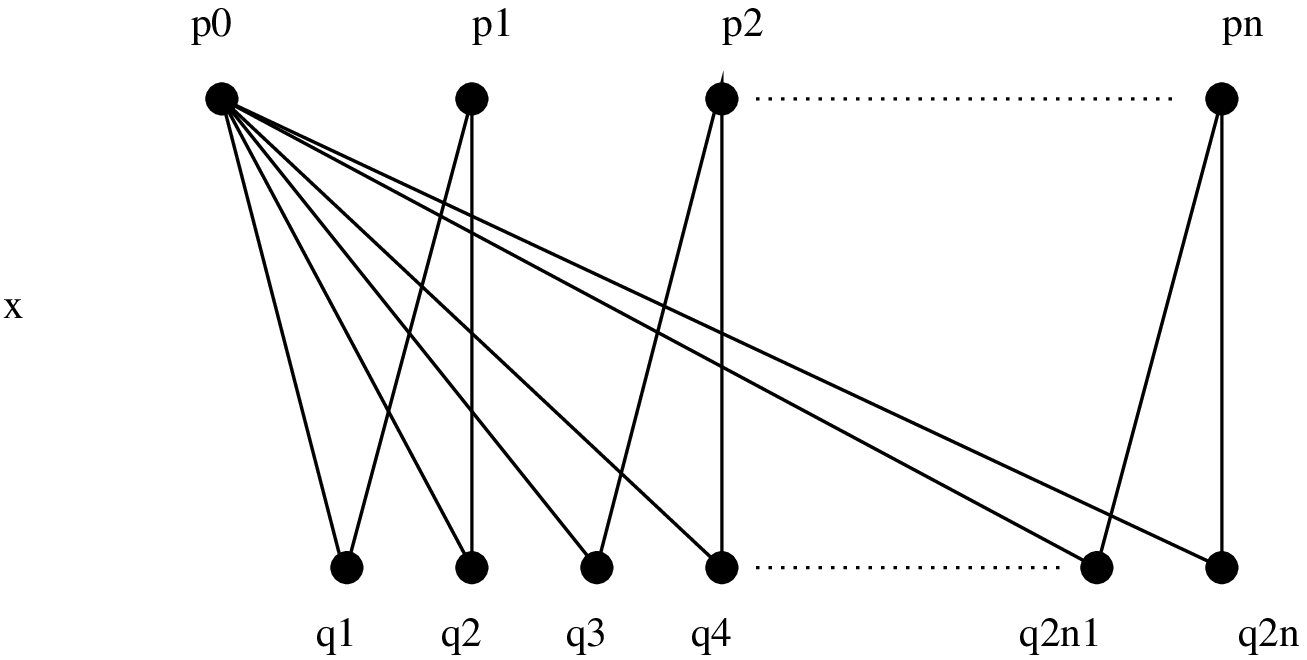,height=3cm}
\end{center}

This minimal space of height $1$ satisfies $Cat_u(X) = gcat(X) = n+1$. However one can show that $Cat(X) = cat(X) = 2$ by using the space $Y = X\cup \{p'_{k}\}_{k=1}^{n-1}$, with the ordering generated by the one of $X$ and $q_{2k}, q_{2k+1} \leq p'_{k} \leq p_0$ for $1\leq k \leq n-1$.
\end{example}

\begin{rem}\label{Ex1mas}
In a similar fashion to Remark \ref{Ex1}, by using \cite{vilches15}[Prop.6.2 and Prop6.5] and \cite{Barmak}[Th.5.2.1], and letting $Cat_s(X)=min\{Cat(sd^k(X)): k\geq 0\}$ and $gcat_s(X)=min\{gcat(sd^k(X)): k\geq 0\}$, we obtain the following diagram of inequalities with similar considerations.
\begin{equation}
\begin{tabular}{ccccccc}
{$cat(X)$}&{$\geq$}&$cat_s(X)$&$\geq$&$cat_w(X)$\\
\rotatebox[origin=c]{90}{$\leq$}&& \rotatebox[origin=c]{90}{$\leq$}&&\rotatebox[origin=c]{90}{$\leq$}\\
&&&&&&\\
$Cat(X)$&$\geq$&$Cat_s(X)$&$\geq$&$Cat(|\mathcal O(X)|)$\\
\rotatebox[origin=c]{90}{$\leq$}&& \rotatebox[origin=c]{90}{$\leq$}&&\rotatebox[origin=c]{90}{$=$}\\
&&&&&&\\
$gcat(X)$&$\geq$&$gcat_s(X)$&$\geq$&$Cat(|\mathcal O(X)|)$
\end{tabular}
\end{equation}
\begin{center}
	{\bf Open Question:} Find a finite $T_0$-space $X$ holding strict inequalities in some (or any) of the lines and/or columns of the diagram above.
\end{center}
\end{rem}

\begin{rem}
A strong LS type parameter specially devised for
finite spaces and its combinatorial counterpart in the class of cell complexes are defined in \cite{Tanaka}.
\end{rem}

\section{Numerical invariants and the maximal set of a finite space}\label{princat}

Another distinctive property of a finite space $X$ is the existence of its set of maximal points, $Max(X)$. Notice that for any beat point $x\in X$, one gets $Max(X-\{x\}) \leq Max(X)$ and so $Max(X_0) \leq Max(X)$ for the core of $X$, $X_0$. Moreover, as $X_0$ is determined up to homeomorphism by the homotopy type of $X$, the cardinal number $|Max(X_0)|$ is a numerical homotopy invariant of $X$. Furthermore, as minimal open sets in a finite $T_0$-space $X$ are contractible in themselves, we get
$$cat(X) \leq Cat(X) \leq Cat_u(X) = gcat(X_0) \leq |Max(X_0)| \leq |Max(X)|.$$
\par
The maximal set of a finite space yields a special type of open sets. Namely, we define a \emph{prime open set} in $X$ as an open set
$U_J = \cup _{x\in J}U_x$, where $J \subseteq Max(X)$. Then the following lemma holds.

\begin{lem}\label{lemaux}
Any open cover $\mathcal{U}$ of $X$ admits a refinement $\mathcal{V}$ consisting of prime open sets  with $|\mathcal{V}|\leq |\mathcal{U}|$.
\end{lem}

\begin{proof}
For each $U$ let $V_U$ be the (possibly empty) open set
$V_U= \cup \{U_x; x\in Max(X)\cap U\}$. As any $x\in Max(X)$ belongs to some $U\in \mathcal{U}$, the family  $\mathcal{V}$ of non-empty sets in $\{V_U; U \in \mathcal{U}\}$ is the required refinement.
\end{proof}

This immediately implies:

\begin{lem}\label{lemaux1}
For any finite $T_0$-space $X$, $cat(X)$  is the least cardinal number of those coverings of $X$ consisting of prime open sets contractible in $X$.
\end{lem}

Lemma \ref{lemaux1} does not hold for the geometric category since subsets of contractible sets need not be contractible. For instance, for the space $X$ in Example \ref{Expalo}(1), the minimal open set $U_a$ contains the non-contractible open set $U_a-\{a\}$.

As observed in Example \ref{Expalo}(2) above, the geometric category of a space can be altered by introducing up beat points and considering non-prime open sets. It is then natural to state a specialized version of the geometric category by restricting to those coverings consisting of prime open sets. More precisely,

\begin{defin}
The \emph{prime geometric category} of a finite $T_0$-space $X$ is the
least cardinal number, $gcat_p(X)$, of coverings of $X$ by prime open sets which are contractible in themselves.
\end{defin}

\begin{example}\label{Catuygpcat}
For each $t\geq 1$  we now construct a space $X(t)$ with $gcat(X(t)) \neq gcat_p(X(t))$. In fact it shows that the gap between $gcat$ and $gcat_p$ can be arbitrarily large.

\begin{center}\label{dibgcat10b}
\psfrag{d1}{$d_1$}\psfrag{dt}{$d_t$}\psfrag{e1}{$e_1$}\psfrag{et}{$e_t$}\psfrag{f1}{$f_1$}\psfrag{ft}{$f_t$}\psfrag{k1}{$k_1$}\psfrag{kt}{$k_t$}\psfrag{a}{$a$}\psfrag{b}{$b$}\psfrag{c}{$c$}\psfrag{g}{$g$}\psfrag{h}{$h$}\psfrag{x}{$X(t)\equiv$}
\psfig{figure=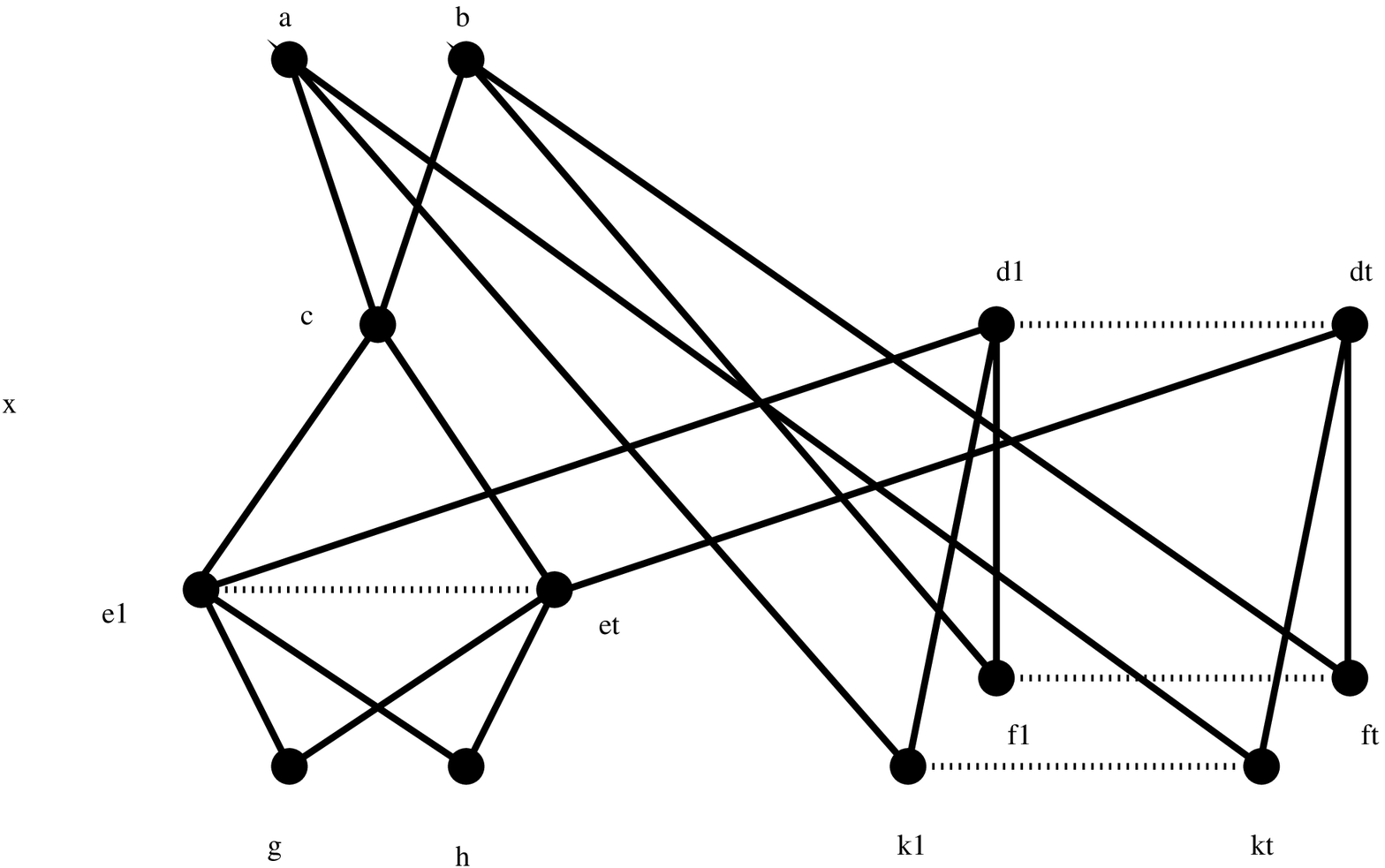,height=4cm}
\end{center}
It is clear that $X(t)$ is a minimal space; moreover, the open set $U_a\cup U_c\cup U_{d_1}\cup \dots \cup U_{d_t}$ retracts onto $U_c$ and so $gcat(X(t)) = Cat_u(X(t)) = 2$. However, $gcat_p(X(t)) = t+1$ since the union $U_a\cup U_b$ is the only prime open set which is contractible in itself and is not minimal.

\end{example}

In contrast to ordinary geometric category the prime version is a homotopy invariant. More precisely,

\begin{prop}\label{Cat= gcat}
Let $X$ and $Y$ be finite $T_0$-spaces. Then the equality $gcat_p(X) = gcat_p(Y)$ holds whenever $X$ and $Y$ are homotopy equivalent.
\end{prop}

\begin{proof}
We can assume that $Y$ is the core of $X$ and so homeomorphic to the core of any finite space with the homotopy type of $X$. By an inductive argument, we can assume that $X-Y$ reduces to a beat  point $b$. Moreover, by Proposition \ref{removingdown} we can assume that $b$ is an up beat point (whence, $b\notin Max(X)$). Let $d$ be the minimum of the points greater than $b$ and $\mathcal{V} = \{V_1, \dots, V_n\}$ be a covering of $Y$ consisting of contractible prime open sets. Take $V_i$ with $d\in V_i$; then the minimal open set $U_d$ of $d$ in $Y$ is contained in $V_i$. It is immediate that $U'_d = U_d\cup\{b\} $ is the minimal open set of $d$ in $X$ and $b$ is an up beat point in $V'_i = V_i \cup \{b\}$. Thus,  $\mathcal{V}' = (\mathcal{V} -\{V_i\}) \cup \{V'_i\}$ is a covering of $X$ by contractible prime open sets, and therefore $gcat_p(X) \leq gcat_p(Y)$.
\par
Conversely, given a covering of $X$ consisting of contractible prime open sets $\mathcal{W} = \{W_1, \dots, W_m\}$, let $d$ denote again the minimum of the points above $b$ in $X$. Then $b$ is an up beat point in any $W_i$ with $d\in W_i$, and $W'_i = W_i-\{b\}$ is a contractible prime open set in $Y$ since $b$ is not in $Max(X)$. Thus, $\mathcal{W}' = \{W'_j; d\in W_j\} \cup \{W_j; d\notin W_j\}$ is a covering of $Y$ consisting of contractible prime open sets of $Y$. We have proved $gcat_p(Y) \leq gcat_p(X)$, and hence $gcat_p(X) = gcat_p(Y)$.
\end{proof}

\begin{rem}{\rm
Given a finite $T_0$-space $X$, if infinite Alexandrov spaces are allowed to represent the homotopy type of $X$  we do not know whether Proposition \ref{Cat= gcat} holds.
}
\end{rem}

As a consequence of Proposition \ref{Cat= gcat}, the prime geometric category of a finite $T_0$-space $X$ is a new numerical homotopy invariant which bounds $Cat_u(X)$ from above, since the equality $gcat_p(X) = gcat_p(X_0)$ holds for the core $X_0$ of $X$. Then we have:

$$Cat_u(X) = gcat(X_0) \leq gcat_p(X_0) = gcat_p(X).$$

Notice that the prime LS-category of $X$ coincides with the usual LS-category by Lemma \ref{lemaux1} and the following inequalities relate the numerical invariants defined above.
\begin{equation}\label{ineq}
cat(X) \leq Cat(X) \leq Cat_u(X) \leq gcat_p(X) \leq |Max(X_0)| \leq |Max(X)|.
\end{equation}

\begin{example}\label{ex1}
{\rm
For each $n\geq 2$, let $X_n$ be the finite space of height 1 in Example \ref{Ex1}. It is readily checked that $cat(X_n) = gcat_p(X_n) = |Max(X_n)| = n$ . Incidentally, $X_n$ is the smallest finite space with LS-category $n$. More generally, the smallest finite space of height $k$ with LS-category $n$ is the finite space whose Hasse diagram has two elements at each level $\leq k-1$ and $n$ maximal elements; see Figure \ref{figLUZ1-2}
\begin{figure}[ht]
	\centering
	\psfrag{a0}{$a_0$}\psfrag{b0}{$b_0$}\psfrag{x1}{$x_1$}\psfrag{x2}{$x_2$}\psfrag{x3}{$x_3$}\psfrag{xn2}{$x_{n-2}$}\psfrag{xn1}{$x_{n-1}$}\psfrag{xn}{$x_{n}$}\psfrag{l}{$\cdots$}\psfrag{d}{$\vdots$}\psfrag{a1}{$a_1$}\psfrag{b1}{$b_1$}\psfrag{ah1}{$a_{h-1}$}\psfrag{bh1}{$b_{h-1}$}\psfrag{ah2}{$a_{h-2}$}\psfrag{bh2}{$b_{h-2}$}\psfrag{ah3}{$a_{h-3}$}\psfrag{bh3}{$b_{h-3}$}		\psfig{figure=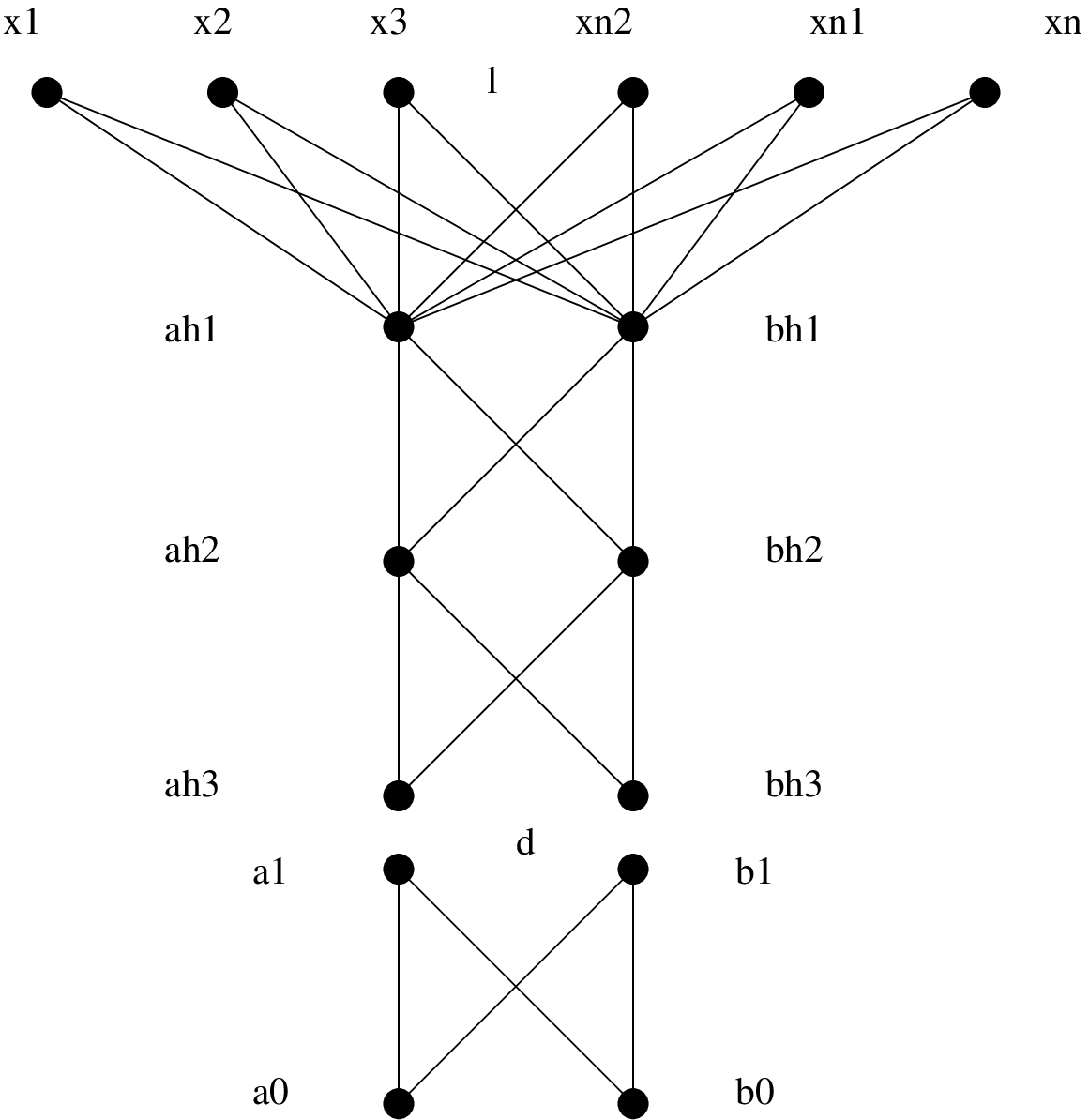,height=4cm}
	\caption{}
	\label{figLUZ1-2}
\end{figure}
}
\end{example}

The problem of measuring the gaps in the sequence of inequalities in (\ref{ineq}) arises naturally. We have already observed in Examples \ref{CatyCatu} and \ref{Catuygpcat} that the intervals $[Cat(X), Cat_u(X)]$ and $[Cat_u(X), gcat_p(X)]$ can be arbitrarily large. It is obvious that for a height $1$ space $X$ whose Hasse diagram is a cycle and $|Max(X)| = k+2$ the interval $[gcat_p(X), |Max(X)|]$ has length $k$.
\par
{\bf Open Question:} Estimate the interval $[cat(X), Cat(X)]$ for finite $T_0$-spaces.

Recall that a well-known result due to Fox \cite{CLOT03} states that $Cat(X) \leq cat(X)+1$ for a CW-complex $X$. In the next section we will show that $cat$ and $Cat$ coincide in the class of height 1 spaces.

\begin{table}
\begin{center}
\begin{tabular}{| c | c |}
\hline
Notation & Definition \\ \hline
 $cat$  & Minimal number of contractible open sets in $X$ that cover $X$ \\  \hline

 $gcat$  & Minimal number of contractible open sets that cover $X$ \\  \hline
  $Cat$  & Minimal value of $gcat$ in the homotopy type of $X$ \\  \hline

$cat_w$  & Value of $cat$ of the underlying polyhedron to the order complex of $X$ \\ \hline

$cat_s$  & Minimal value of $cat$ on the successive barycentric subdivisions of $X$ \\ \hline

$Cat_u$  & Value of $gcat$ on the core of $X$ \\ \hline

 $gcat_p$  & Minimal number of contractible prime open sets that cover $X$ \\  \hline

\end{tabular}
\caption{Different versions of the category}
\label{tab:coches}
\end{center}
\end{table}

\section{Spaces of height 1}\label{height2}

Spaces of height 0 are exactly finite discrete spaces and so their LS-category is trivially the number of points of the spaces. Also any positive integer can be realized as any of the LS-type numbers defined above of a finite space of height 1; see Example \ref{ex1}.

We will next show that the strong category equals the LS-category in the class of height 1 spaces. Namely,

\begin{theor}\label{cat=Cat1}
Let $X$ be a finite $T_0$-space of heigth $1$. Then $Cat(X) = cat(X)$. Moreover, the equaility is
achived by a height $2$ space homotopy equivalent to $X$.
\end{theor}
\begin{remark}
	Actually, the result is valid for any finite $T_0$-space whose core has height $1$.
\end{remark}
The proof of Theorem \ref{cat=Cat1} is a consequence of the two following lemmas.
\begin{lem}\label{lem1}
Let $X$ be a connected finite space of height 1. Given $J\subseteq Max(X)$, the prime open set $U_J$ is contractible in $X$ if and only if each component of $U_J$ is contractible in itself.
\end{lem}

\begin{proof}
Obviously $U_J$ is contractible in $X$ if all its components are contractible in themselves. Conversely, let $C\subseteq U_J$ be a component. Then $C = U_{J'}$
is also a prime open set for some $J'\subseteq J$ and contractible in $X$.
\par
Let $C_0 \subseteq C$ be the core of $C$. Next, we show that $C_0$ reduces to a point and so $C$ is contractible.
\par
As $C_0$ remains contractible in $X$, there exists a sequence $f_0, f_1, \cdots, f_n: C_0 \to X$, of comparable maps with
$f_0$, the inclusion $C_0 \subseteq X$, and a constant map $f_n$. Moreover, we can assume  that there exist points $x_0, \dots, x_n\in C_0$ such that $f_{i-1} =f_i$ on $C_0-\{x_{i-1}\}$ and $ f_{i-1}(x_{i-1}) < f_i(x_{i-1})$ or $f_{i-1}(x_{i-1}) > f_i(x_{i-1})$ ($1 \leq i \leq n -1$).
\par
If $C_0$ were not just a point, we arrive to a contradiction.
\par
First assume that $x_0$ is at level $0$. Then $x_0 = f_0(x_0) < f_1(x_0)= m \in Max(X)$.  Moreover, by connectedness $\{y \in C_0; y > x_0\} \neq \emptyset$, and there exists $m' \in Max(C_0) \subseteq Max(X)$ such that $x_0 \leq m'$ and $m \neq m'$, since otherwise $x_0$ would be an up beat point in $C_0$. Then by continuity we reach the contradiction $m = f_1(x_0) \leq f_1(m') = f_0(m') = m'$.
\par
Assume now that $x_0$ lies at level $1$. Then $f_1(x_0) \in Min(X)$ and $y =f_1(x_0) < f_0(x_0) = x_0$. Again by connectedness
$Min(C_0) \cap U_{x_0} \neq \emptyset$, and there exists $y' \in Min(C_0) \subseteq Min(X)$ with $y'< x_0$. Moreover $y' \neq y$, since otherwise $x_0$ would be a down beat point in $C_0$. By continuity, $y = f_1(x_0) \geq f_1(y') = f_0(y') = y'$, which is contradiction.
\end{proof}

\begin{lem}\label{lem2}
Let $\{U_1, \dots, U_n\}$ be an open covering of a finite connected $T_0$-space $X$ such that $U_1$ decomposes in $s$ connected components which are contractible in themselves.
Then there exists a finite $T_0$-space $Y$ and an open covering $\{U'_1,U'_2, \dots U'_n\}$ of $Y$ such that $X$ and $U_i$ are
deformation retracts of $Y$ and  $U'_i$, respectively, for $2\leq i \leq n$, and $U'_1$ consists of $s-1$ connected components contractible in themselves.
Moreover, $height(Y)=max\{ 2,height(X)\}$.
\end{lem}

\begin{proof}
Firstly we observe that all minimal open sets in a finite $T_0$-space meets $Min(X)$ and moreover, given any two minimal elements $p,q\in Min(X)$ there exists a path in $X$
\begin{equation}\label{goodpath}
p = x_0 \leq x_1\geq x_2 \leq \dots \leq x_{2m-1}\geq x_{2m}= q
\end{equation}
with $x_{2k} \in Min(X)$ and $x_{2k-1}\in Max(X)$ for all $k\leq m$.
Indeed, given any path $L$ between $p$ and $q$: $L \equiv p =y_0 < y_1 > y_2 \dots < y_{2m+1} > y_{2m} = q$,
choose $x_{2k}\in Min(X) \cap U_{y_{2k}}$ ($0\leq k \leq m$).
Notice that $x_0 = y_0 = p$, $x_{2m} = y_{2m} = q$ and the desired path in obtained by replacing
the $y_{2k}$'s by $x_{2k}$'s in $L$ since $y_{2k+1}$ remains greater than both $x_{2k}$ and $x_{2k+2}$ for all $0\leq k \leq m$.
Finally we replace each $y_{2k-1}$ by some $x_{2k-1}\in\overline{\{y_{2k-1}\}}\cap Max(X)$, $1\leq k\leq m$.
\par
Assume that $\{C_j\}_{l=1}^s$ is the family of connected components of $U_1$. Let $\Gamma$ be a path as in (\ref{goodpath}) joining some $p\in Min(X)\cap C_1$ to some $q\in Min(X) \cap C_l$ for some $2\leq l \leq s$. We can assume that $\Gamma$ is an arc (that is, $x_i\neq x_j$ if $i\neq j$) by removing all cycles in $\Gamma$.
\par
We can also assume without loss of generality that $\Gamma \cap U_1$  reduces to $\{p, q\}$ since otherwise we can consider the
subarc $\Gamma' \subseteq \Gamma$  running from the last element of $\Gamma$ in $Min(X) \cap C_1$ to the first element of
$\Gamma$ in $Min(X) \cap (U_1-C_1)$. In particular, $x_{2k-1}\notin U_1$ for all $1\leq k\leq m$.
\par
For each $0\leq k \leq m-1$, we will write $j=j(k)$ if $x_{2k+1}\in U_{j}$, with $j=2,\dots, n$, and take some point $q_k \notin X$.
Consider the space $Y= X\cup \{q_k\}_{k=1}^m$ with the ordering generated by the one on $X$ and $x_{2k}, x_{2k+2} < q_k < x_{2k+1}$  ($0\leq k \leq m-1$).
Notice that  the minimal open sets of $q_k$ and $x_{2k+1}$ in $Y$ are $U^Y_{q_k} = \{q_k, x_{2k}, x_{2k+2}\}$ and $U^Y_{x_{2k+1}} = U_{x_{2k+1}}\cup \{q_k\}$.
This determines a new arc $\widehat{\Gamma}\equiv  p =y_0 < q_1 > y_2 \dots < q_{2m+1} > y_{2m} = q$ between $p$ and $q$.
Notice that $\widehat{\Gamma}$ is open in $Y$ and, clearly, contractible.
\par
Let $U'_1 = U_1\cup \widehat{\Gamma}$ and $U'_j = U_j \cup \{q_k; j= j(k)\}$ ($2\leq j \leq n$). These sets are open sets in $Y$  and all the points $q_k$ are
up beat points in $Y$ and those $q_k$ with $j= j(k)$ are, in addition, up beat points in $U'_j$.
Hence $Y$ and $U'_j$ retract onto $X$ and $U_j$, respectively. Moreover, the family of connected components of $U'_1$ consists
of the union $C_1\cup \widehat{\Gamma} \cup C_l$ and the former components $C_i$ with $i\neq 1,l$.
\par
Finally, as $(C_1\cap C_l)\cap \widehat{\Gamma} = \{p,q\}$ and $C_1$, $C_l$ and $\widehat{\Gamma}$ are contractible, so is the component $C_1\cup \widehat{\Gamma} \cup C_l$.
\par
Finally we observe that, since the $x_{2k-1}$'s are maximal elements, one gets $height(Y)= height(X)$ if $height(X) \geq 2$ and $height(Y)= 2$ otherwise.
\end{proof}

\begin{proof} (Theorem \ref{cat=Cat1})
Assume $cat(X) = n$ and let $X = \cup_{j=1}^n U_j$  be an open covering where each $U_j$ is contractible in $X$. Then Lemma \ref{lem1} yields that each $U_j$ is a disjoint union of open sets contractible in themselves.
\par
By applying Lemma \ref{lem2} recursively on the number of connected components of $U_1$ we construct a finite space $X_1$ with the same homotopy type as $X$ such that $X_1 = \cup_{j=1}^n U^1_j$ where $U^1_j$ is contractible and each $U^1_j$ is homotopy equivalent to $U_j$ for $j\geq 2$.
In particular each $U^1_j$ has the same number of (contractible) components as $U_j$ for $j\geq 2$.
\par
 Now one applies recursively Lemma \ref{lem2} to $U^1_2$ in $X_1$ to get a new space $X_2$ with the same homotopy type as $X_1$ such that $X_2 = \cup_{j=1}^n U^2_j$ is an open covering of $X_2$ where $U^2_j$ is contractible and $U^2_j$ is homotopy equivalent to $U^1_j$ whenever $j\neq 2$. In particular $U^2_1$ is contractible.
\par
This way we proceed inductively on $n$ to get a space $X_n$ within the homotopy class of $X$ such that $X_n$ admits an open covering $\{U^n_j\}_{j=1}^n$ where each $U^n_j$ is contractible.
Hence $Cat(X) \leq n = cat(X)$ and so $Cat(X) = cat(X)$.
As observed Lemma \ref{lem2}, we have $height(X_n) = \max\{2, height(X_{n-1})\} = \dots = \max\{2,height(X)\} = 2$.
\end{proof}

\begin{rem}\label{crown}
{\rm
As an immediate consequence of Lemma \ref{lem1} we get that an open set $U$  is contractible in a finite space of height 1  if an only if it does not contain a cycle, that is, there is no sequence of the form $x_1,y_1, x_2, y_2, \dots, x_n$ with $x_1 = x_n$ and $x_i > y_i < x_{i+1}$ for $1\leq i \leq n-1$.
\par
Cycles are special instances of the configurations termed crowns; see \cite{Rival}[Th. on p. 310]and \cite{DuffusRival}.
}
\end{rem}

As mentioned in Section \ref{notation}, a lower bound for the LS-category of a finite $T_0$-space is given by the  simplicial category of its order complex $\mathcal{O}(X)$, $scat(\mathcal{O}(X))$. Moreover, if $K$ is a graph, then $scat(K)$ is equal to the vertex arboricity $va(K)$ of $K$, see \cite{vilches15}. Recall that given a graph $G$ the \emph{vertex arboricity} $va(G)$ of $G$ is the minimum $n$ for which the vertex set of $G$ admits a partition $P_n=\{V_1,\ldots,V_n\}$ such that for every  $i=1,\dots,n$  the subgraph of $G$ induced by the vertices of $V_i$ is acyclic.

On the other hand, an upper bound for the LS-category of finite spaces of height 1 can be obtained in terms of vertex arboricity for multigraphs.

Given a space $X$ of height 1, let us define the multigraph $\Gamma(X)$ associated to $X$ in the following way: the vertex set of $\Gamma(X)$ is $Max(X)$; a simple edge exists between any two vertices $v, w\in Max(X)$ if the corresponding minimal open sets $U_v$ and $U_w$ have exactly one common neighbour; and a double edge connects two vertices $v, w \in Max(X)$ if the corresponding minimal open sets $U_v$ and $U_w$ have more than one common neighbour.

A description of the LS-category in the terms of graph arboricity is accomplished by a sort of relative arboricity as follows. Given a graph $G=(V,E)$, a \emph{domination set} of $G$ is a set $D\subseteq V$ such that each vertex in $V\backslash D$ is adjacent to some vertex of $D$. Then the \emph{$D$-arboricity $a_D(G)$ of $G$}, is the minimum $n$ for which $D$ admits a partition $\{D_1,\ldots,D_n\}$ such that for every  $i=1,\dots,n$  the subgraph of $G$ given by the union of the edges of $G$ incident to the vertices in $D_i$ is acyclic. The previous remarks are summarized in the following proposition:

\begin{prop}\label{prop1}

The following (in)equalities hold for any finite space of height 1,
$$
va(\mathcal{O}(X))\leq cat(X)=a_{Max(X)}(\mathcal{O}(X)) \leq va(\Gamma(X)).
$$

\end{prop}

\begin{proof}
Observe that the first inequality is already known. The equality in the middle follows from the fact that if $U_J$ is a prime open set, then the corresponding subgraph $G_J \subseteq \mathcal{O}(X)$ generated by the edges incident at vertices of $J$ is acyclic if and only if each component of $U_J$ is contractible. We conclude by Lemma \ref{lem1}.

To check the second inequality, let $\{J_1, \dots, J_n\}$ be a partition of  $Max(X)$ such that $\Gamma_i(X) =
\cup_{j\in \Lambda_i} T_j$ is a forest. Let $V_j$ be the vertex set of the tree $T_j$, and $U_j = \cup_{x\in V_j} U_x$. Now observe that the open sets $U_i = \cup _{j\in\Lambda_i} U_j$ do not contain cycles (see Remark \ref{crown}). Indeed, any such cycle  yields a cycle in $T_j$, in the classical graph-theoretic sense. Therefore, $U_i$ is contractible in $X$ by Remark  \ref{crown}, and then $cat(X) \leq va(\Gamma(X))$.
\end{proof}

The upper bound provided by the multigraph $\Gamma(X)$ may not be very accurate, as the following example shows:

\begin{example}{\rm
For $n = 2k$, let $X$ be the finite space of height 1 with $ A = \{a_0, a_1, \dots a_n\}$ as points at level $0$, $\{b_0, \dots, b_n\}$ as points at level $1$, and such that
$U_{b_0} = \{b_0\}\cup A$, while $U_{b_i} = \{b_i, a_0,a_i\}$ for $i\geq 1$. Then $U_1 = \cup_{i=1}^n U_{b_i}$ and $U_2 = U_{b_0}$ are contractible open sets such that $U_1\cup U_2=X$, so that $ cat(X) = 2$. In contrast, no set with three or more $b_i$ ($i\geq 1$) as well as no set containing $b_0$ and some $b_i$ ($i\neq 0$) generates an acyclic graph in $\Gamma(X)$, and so $va(\Gamma(X)) = k +1$.
}
\end{example}

\section{A contractibility algorithm}

\label{compalg}

 In this section we present an algorithm to obtain the core of a finite $T_0$-space $X$, and in particular deciding whether $X$ is eventually contractible or not. The algorithm is based on the following basic observation which is  an immediate consequence of the homotopical properties of beat points.

\textbf{Fact.} If $X$ is a contractible space and $x$ is a beat point of $X$, then $X-\{x\}$ is contractible.

We next describe the procedure.

 \textsc{Algorithm:} \textsc{Contractibility Test (CT)}

 \textbf{Input:} A finite $T_0$-space $X$ as a poset.

 \textbf{Output:} The core of $X$ and the answer to the question ``{\it Is $X$ contractible?}"

\begin{enumerate}
	
	\item Identify all the beat points of $X$.
	
	\item If $X$ has no beat points and is itself a point, STOP. Otherwise, go to 3).
	
	\item If $X$ has no beat points and is not a point, STOP (and we have reached the core of $X$). Otherwise, go to 4).
	
	\item Remove the beat points identified in 1).
	
	\item Rename the obtained space as $X$.
	
	\item Go to 1).
	
\end{enumerate}

  As we are dealing with finite spaces, the algorithm terminates. If it terminates in some iteration at Step 2, the original space was contractible, while if it stops in some iteration at Step 3, the original space was not. Next we prove that this procedure finishes in polynomial time.

 \begin{lem}
 	For a finite space $X$ of cardinal $n$, obtaining its core and detecting its contractibility has time complexity  of order ${\mathcal O}(n^4)$.
 \label{comple}
 \end{lem}

 \begin{proof}
  Firstly, the {\it transitive reduction algorithm} described in \cite{Gries} must be performed in order to obtain a  Hasse diagram of the space $X$.	Let $D=(V,E)$, $V=\{v_1, \, v_2, \dots, v_n\}$,  be the digraph representing the Hasse diagram of $X$. 
  By  abuse of notation, we say that a vertex $b$ is a  beat point   of $D$ if it corresponds to a beat point of $X$. A vertex $b$ of $D$ is 
  a beat point if and only  if $b$ has indegree or outdegree (possibly  both) equal to one in $D$.
 	
 	Let us consider the adjacency  matrix $A=(a_{ij})_{n\times n}$ of $D$, defined by $a_{ij}=1$ if $(v_i, v_j)\in E$ and $a_{ij}=0$ in other case. For each beat point $b$ with indegree $\delta_i(b)=1$ and outdegree  $\delta_o(b)=k>0$ let us consider the edge $(a,b)$ providing the only in-edge.  Removing a beat point produces a sequence of operations over the matrix $A$ that consists of:
 	\begin{enumerate}
 		\item[(1)] Replacing the row of $a$ by the addition of the rows of $a$ and $b$  
 	and any out-edge $(b,v_j)$    by a new out-edge $(a,v_j)$. 
 	\item[(2)]  Deleting the column and the row corresponding to  $b$.
 	\end{enumerate}
 For each beat point $b$ with  $\delta_o(b)=1$ and  $\delta_i(b)=h>0$ an analogous method is used, by interchanging the roles of columns and rows in operations (1) and (2).
 	
 	This way, after the removal of all beat points of $D$, a new acyclic digraph $D'=(V', E')$ is obtained verifying $|V'|\leq |V|-1$ and $|E'|\leq |E|-1$. Two situations are possible:
 	\begin{enumerate}
 		\item $D'$ is the Hasse diagram of a new space $X'$, that is, the transitivity property of $X'$ is reflected in $D'$ ($(a,b), (b,c) \in E'$ implies $(a,c)\notin E'$).
 		\item $D'$ is not the Hasse diagram of any space since there are three edges $(a,b), (b,c), (a,c) \in E',$  for some vertices $a, b, c \in V'.$
 	\end{enumerate}
 	
 	In (1) new beat points must be searched for $D'$, while in (2)   the {\it transitive reduction algorithm} must be performed again in order to obtain a new digraph corresponding to a Hasse diagram of a new space $X'$.
 	
 	Now, we study the time complexity of the entire procedure. Detecting the set of beat points of $D$ can be done in     ${\mathcal O}(n^2)$  time, since it suffices to find out in the matrix $A$ arrows with only  one 1  and columns with only  one 1.
 	It is clear that, for a beat point $b$,  operations (1) and (2) have time complexity of constant order. 
 	According to  \cite{Gries}, the time complexity  for obtaining the transitive reduction of an acyclic digraph is ${\mathcal O}(n^3)$, being $n$ the order of the digraph.
 	In the worst case, only one beat point is eliminated in each step, and the vertex set decreases by one after each removal. Hence we can conclude that the total time complexity in the worst case is of order $\sum_{k=1}^{n} k^3 =\frac{n^2(n+1)^2}{4}$, that is ${\mathcal O}(n^4).$
 \end{proof}

\section{An algorithmic approach to the geometric category}
\label{identify}

In this section we explore possible algorithmic treatments of the geometric category and the other related numerical parameters $Cat_u$ and $gcat_p$ for a finite $T_0$-space $X$. With this purpose we introduce the following definition: a subset $J \subset X$ is said to be \emph{compatible} if the open set $U_J = \cup_{x\in J} U_x$ is contractible.
\par
We next proceed to analyze the compatibilities. By using the CT-algorithm, which allows us to decide whether $J$ is compatible or not. Moreover, in order to determine $Cat_u(X)$ we first apply the CT-algorithm to find the core $X_0$ of $X$. For the prime strong geometric category $gcat_p(X)$, we first reduce $X$ to its set of maximal elements and then analyse the compatibilities among them.
\par

There are two possible approaches when undertaking this calculation. The first one consists in identifying the compatibility (or not) of \emph{all} subsets of $X$ for $gcat(X)$, $Max(X)$ for $gcat_p$ and $Max(X_0)$ for $Cat_u(X)$,  and then deriving from this information the exact value of the corresponding numerical parameter (or, equivalently, the covering number of an appropriate hypergraph associated to $X$ in Section \ref{monocomp}). We describe below greedy algorithms that work in this context. Unfortunately, there will be instances of spaces for which these procedures will need exponential time to solve the problem, although they should be useful for ``small" cases. On the other hand, we may investigate bounds for these numerical parameters without computing the compatibility (or not) of \emph{all} subsets of $X$ ($Max(X)$ or $Max(X_0)$, respectively) but only of a number of them that involve a polynomial number of operations. This can be achieved by restricted applications of the mentioned algorithm, as we will also see below.

We will proceed to describe the algorithm for the geometric category $gcat(X)$. In a similar way we can deal with $gcat_p(X)$ ($Cat_u(X)$, respectively) simply by taking below the smaller set $Max(X)$ ($Max(X_0)$, respectively) in the place of $X$ and prime open sets in the place of all open sets of $X$ ($X_0$, respectively).

Observe that, for $|X| = n$, there are $2^n-1$ non-empty different collections of such open sets. As singletons are clearly compatible, in principle we should take account of $2^n-n-1$ collections.

\subsection{Greedy strategies}
\label{greedy}

Given $k\geq 2$,  we present two systematic greedy procedures to check the compatibility of collections of $k$   minimal open sets  that at the same time provides upper bounds for $gcat(X)$.
\par
\noindent{\bf Notation.} For simplicity, given a finite $T_0$-space with $|X| = n$, we will write $X = \{1,\dots, n\}$ so that for $J=\{i_1, i_2, \dots, i_k\}$ with $1\leq i_1<i_2<\dots <i_k\leq n$, we denote $U_J=\cup_{i\in J}U_i$ simply by  $(i_1, i_2, \dots, i_k)$.  The number $k$ will be called the \emph{length} of the collection $J$.

Let us now describe the first algorithm, which we will call \emph{U-algorithm}; the name comes from the fact that we check sets of prime open sets with growing length, ``going up".

 \textsc{$U$-Algorithm:}

\textbf{Input:} The family of minimal open sets $U_J$ with  $J \subset X$,  $|X|=n$.

\textbf{Output:} An upper bound for $gcat(X)$, or $gcat(X)=n$.

\begin{enumerate}

\item Set $k=2$; $bound=n$.

\item For  $k\leq n$ check {\sc CT} for the $\binom{n}{k}$ different collections of minimal open sets of length $k$.

If at least one is compatible, then set $bound=n-k+1$, $gcat(X)\leq bound$

else $gcat(X)\leq bound.$

 $k=k+1$.

\item If $bound < n$ set  $gcat(X)\leq bound $ and STOP

else set  $gcat(X)=n$ and STOP.

\end{enumerate}

\begin{remark}
If there is a collection $C=\{i_1,\dots,i_k\}$ of length $k$ that is compatible, then the covering given by $(i_1,\dots,i_k)$ and the remaining open sets assures that $gcat(X)\leq n-k+1$. The algorithm does not give in general an exact value for the $gcat(X)$, because the coverings whose length gives $gcat(X)$ could not be identified by the algorithm. The exceptions are the cases $k=n-1$ and $k=n$, where the equality $gcat(X)=2$ or $gcat(X)=1$  is immediately deduced from the corresponding inequality.
\end{remark}

Note that a dual algorithm (\textsc{$D$-Algorithm}) is possible, starting from the unique collection of length $n$, continuing with the collections of length $n-1$ and so on. The output in this case in an upper bound for the category. Although the algorithms are basically symmetric, there is an difference between them. In the $U$-algorithm, every step can improve the bound obtained in the previous one. In the dual version, however, the occurrence of a value for the bound of the category in some step implies immediately that the algorithm stops, as the bound cannot be improved in subsequent steps.

On the other hand, the compatibility structure of the minimal open sets is completely described when you run until the end any of the two algorithms; this is important in order to get an exact computation of $gcat(X)$ using the tools of hypergraph theory described in last section.

In general, the complexity of the algorithms (when computing the whole compatibility structure) is exponential in the number $n$ of open minimal sets, as $2^n-n-1$ collections of open minimal sets must be checked. However, if we are only interested in bounds for the category and we run the algorithm only until some fixed length, the complexity is \emph{polynomial} in $n$. For a certain length $k$, we will call $U(k)$ and $D(n-k)$ the shortened versions of the $U$-algorithm and its dual, respectively. In the case of the $U(k)$-algorithm, the fact that there are $\binom{n}{k}$ different collections of $k$ minimal open sets and that checking the CT-algorithm for a collection of  $k$ minimal open sets has complexity $O(k^4)$ (Lemma \ref{comple}) implies that the complexity of this algorithm is $O(n^kk^4)$. A similar reasoning for the $D(n-k)$-algorithm gives a complexity of   $O(n^{k+4})$, bigger than the previous one. So from the point of view of the complexity, the $U$-algorithm is better than its dual, the $D$-algorithm.

 It could be also interesting to combine both approaches, starting from the set of all minimal open sets, then checking pairs, then checking collections of length $n-1$, then triplets, and so on. Alternatively, one can start by checking compatibilities over a random sample of collections of minimal open sets, to get a clue over the number of compatibilities: a high number of them would recommend to use the first algorithm, while a low number would indicate that the dual version would be more useful.

\par
\medskip

\subsection{Heuristics}
\label{heur}
In this subsection we suggest heuristic ways of obtaining bounds for the geometric category of $X$. Recall that $X = \{1,\ldots, n\}$ and contractible open sets
are identified as ordered compatible subsets of $X$. The goal is to find a covering of $X$ consisting of a number of contractible open sets as close as possible to $gcat(X)$. We follow the notation described at the beginning of Section \ref{greedy}.

\textbf{Heuristic 1}. First apply the CT-algorithm to $(1,2)$. If this set is contractible, then apply the CT-algorithm  to $(1,2,3)$, and iterate the process until the smallest $k$ such that $\{1,2,3,\ldots,k\}$ is not compatible is found. Then one can proceed in different ways, in order to identify a covering, which in general have linear or at most polynomial complexity. For example, consider the following scheme. The first element of our covering will be $(1,2,3,\ldots,k-1)$. Apply the {\sc CT}-algorithm to $(k,k+1)$. If this set is contractible, apply the {\sc CT}-algorithm to $(k,k+1,k+2)$, and iterate the process until you find the smallest $j$ such that $\{k,k+1,\ldots,(k+j)\}$ is not compatible.  The second element of the covering will then be $(k,k+1,\ldots,k+j-1)$. Now iterate the process starting with the contractibility of $(k+j,k+j+1)$. The process finishes when $\{n\}$ has been assigned to a member of the covering. It is easy to develope variations of this scheme, checking for example $(1,2,\ldots,k-1,k+1)$ after $(1,2,3,\ldots,k-1)$ at first, and/or considering different ways of perform the updating in the covering.

\textbf{Heuristic 2}. A different approach is to fix a certain $k<n/2$, and to check the compatibility of all the subcollections of length smaller or equal than $k$. This can be always undertaken in polynomial time of degree $k$. Observe that $k$ should be independent of $n$ in order to get the polynomial complexity. After all these compatibilities have been checked, we give an order to the set ${\mathcal A}$ of compatible subcollections (where the unitary ones are included), and we call $A_i$ the $i$-th subcollection in this ordered set. Now we can define a covering $\{U_1,\ldots,U_j\}$ in the following way: $U_1=U_{A_1}$; $U_2=U_{A_{i_1}}$ where  $A_{i_1}\in {\mathcal A}$ with smallest subindex such that $A_{i_1}\cap A_1=\emptyset$ (alternatively, $A_{i_1} - A_1\neq \emptyset$); $U_3=U_{A_{i_2}}$ with  $A_{i_2}\in {\mathcal A}$ with smallest subindex such that  $A_{i_2}\cap A_1=A_{i_2}\cap A_{i_1}=\emptyset$ (alternatively, $A_{i_2} -  (A_1\cup A_{i_1}) \neq \emptyset$); and so on. Observe that different orders give rise to different coverings, and that checking all the possible order allows to obtain the covering of smallest cardinal that can be constructed with collections of length smaller or equal than $k$.

Let us also remark that in Heuristic 1 the {\sc CT}-algorithm is applied a linear number of times (precisely $n-1$ times) while in Heuristic 2 the complexity time needed to check the compatibilities is of polynomial order.

\section{Compatibility structures, Boolean functions and hypergraphs}\label{monocomp}

The structures of compatibility defined by the geometric category and the related LS-type numbers $Cat_u$ and $gcat_p$ can be modeled by a Boolean function that sends compatible sets (possibly empty) to $0$ and incompatible sets to $1$. An analogous notion of compatibility can be defined for $cat(X)$ on $\mathcal{P}(Max(X))$ by considering a prime open set $J\subseteq Max(X)$ to be \emph{$cat$-compatible} if $U_J$ is contractible inside of $X$ (see Lemma \ref{lemaux1}).

Observe that if an open set $U$ is contractible inside of $X$, every subset of $U$ is so. Hence, subsets of compatible collections in the sense of $cat$ are always compatible, and the Boolean function for $cat$ is monotonic. This is not the case for the other LS-type numbers since open subsets of contractible open sets need not be contractible (for instance, the open set $U_{d_1}\cup U_{d_2}$ in the space of Example \ref{Expalo}  is not contractible but it is part of the contractible minimal open set $U_b$).
\par
Although the $cat$-compatibility structure yields a monotonic Boolean function, the algorithmic approach in Section \ref{compalg} does not apply to $cat$ since  a starting algorithm similar to the Contractibility Test is quite more complex for the LS-category because checking whether or not a certain subset $U\subseteq X$ is contractible in $X$ requires devising an algorithm which explores all possible homotopies starting from the inclusion of $U$. It is worth pointing out that Tanaka \cite{tanaka2} describes homotopies between finite spaces in terms of relation matrices.
\par
Notice also that the strong category $Cat(X)$ requires, in principle, analyzing the geometric category of all spaces within the homotopy type of $X$.
\par

The following definition provides a general framework in terms of Boolean functions for the LS-type invariants in this paper except for the strong category. Let $\sigma: \mathcal{P}(Z) \to \{0,1\}$ be a (monotonic) Boolean function on a set $Z$. We call $\sigma$-\emph{category} of $Z$ and denote $\sigma$-$cat(Z)$ the smallest integer $n$ for which there exists a covering $\{U_i\}_{i=1}^n$ of $Z$ such that each $U_i$ is $\sigma$-compatible; that is, $\sigma(U_i)=0$.

\par

There is a natural hypergraph arising from a compatibility structure on a finite $T_0$-space $X$. More generally, any Boolean function $\sigma: \mathcal{P}(Z) \to \{0,1\}$ has associated a hypergraph such that $\sigma$-$cat(Z)$ coincides with the covering number of the hypergraph as defined in graph theory.

Recall that given a finite set $V=\{v_1,\dots ,v_r \}$ an \emph{hypergraph} $H$  in $V$ is a family $E=\{E_1,\dots , E_m \}$ of subsets of $V$ so that $E\neq\emptyset$ and $V=\cup_{i=1}^m E_i$. The elements $v_i\in V$ are called \emph{vertices} and the sets $E_j\in E$ are the \emph{hyperedges} of ${H}$. A hypergraph is termed \emph{simple} or a \emph{Sperner hypergraph} if all its hyperedges are maximal, that is,  if $e'\subseteq e\in E$ then $e=e'$.
\par
Given a hypergraph $H=(V,E)$, an \emph{hyperedge covering} of $H$ is a family of hyperedges of $H$ whose union is the vertex set $V$. The \emph{covering number} $\rho(H)$ of $H$ is the minimum cardinality of a hyperedge covering of $H$. When $H$ is not Sperner, we may consider its Sperner subhypergraph $\widehat{H}= (V,\widehat{E})$, where the hyperedges in $\widehat{E}$ are the  maximal hyperedges in $H$. Clearly $\rho(H)=\rho(\widehat{H})$, so when dealing with the covering number we can work with Sperner hypergraphs.
\par
We refer to  \cite{Berge} and \cite{Bretto} for more information.
\par

There is always a hypergraph $H(\sigma) = (V,E)$ canonically associated to any  Boolean function
$\sigma:\mathcal{P}(Z)\to\{0,1\}$, with vertex set $V=Z$ and $E=\{e\subseteq Z:\sigma(e)=0\}$.

\begin{remark}\label{remarkhyper1}
 Notice that for every subset $U\subseteq Z$, $\sigma(U)=0$ if and only if $U=e$ for some $e\in E$. Hence, ${H}(\sigma)$ describes completely the Boolean function $\sigma$. In fact, in this way any hypergraph $H$ describes a Boolean function, denoted by $\sigma(H)$.
Also observe that a hypergraph may be associated to any function $\tau:\mathcal A\to\{0,1\}$ for a family of subsets $\mathcal A\subseteq \mathcal{P}(Z)$.
\end{remark}

Taking into account the previous considerations, the following result is easily checked.

\begin{prop}
	The $\sigma$-category defined by a Boolean function $\sigma$ coincides with $\rho(H(\sigma))$, the covering number of $H(\sigma)$.
\end{prop}

Given any finite $T_0$-space $X$, the previous proposition applies to the Boolean functions $\sigma_{gcat}$ on $Z= X$, $\sigma_{cat}$ and $\sigma_{gcat_p}$ on $Z = Max (X)$ and $\sigma_{Cat_u}$ on $Z = X_0$ (the core of $X$). Therefore, the knowledge of these parameters provide the knowledge of the corresponding covering number of the associated hypergraph  and conversely. It has been established that for a given hypergraph, computing its covering number can be solved in polynomial time by using Boolean functions defined  over it hyperedges set. However, finding a   covering with at most $k$ hyperedges is an $NP$-complete problem whenever $k$ is not fixed, \cite{FiGoPi18}.

 \par
 Nevertheless, whenever the compatibility structure of $Z$ is completely known, the associated  hypergraph $H(\sigma)$ covering number and hence the $\sigma$-category are bounded. For instance, according to  \cite{MoKo04}, for a hypergraph $H$ of $n$ vertices and $m$ hyperedges such that any  vertex lies in at least $b\geq 1$ hyperedges and every hyperedge size is at most $a$, it is verified
  $$\frac{n}{a}\leq \rho(H)\leq \frac{ln (ml/bn)}{ln(1-b/n)}+\frac{m}{b}\sum_{1\leq j\leq l}\frac{1}{j}$$
  for any integer $l$. Other  bounds for some particular situations can be found in \cite{MoKo04}.

\par
The covering number can be reformulated as in terms of the transversal number of a hypergraph, which we now define. Recall that, given a hypergraph $H = (V,E)$,  a set of vertices $B\subset V$ is said to be a \emph{transversal} if it meets every hyperedge $e\in E$. The \emph{transversal number} $\tau(H)$ is the minimum cardinality of a transversal. Each transversal of $H$ corresponds a covering of the dual hypergraph $H^d=(V^d,E^d)$ of $H$, and viceversa. Recall that the \emph{dual hypergraph} $H^d$ has $E$ as its vertex set, $V^d=E$, and a hyperedge in $E^d$ is a family $e^d_v = \{e\in E/v\in e\}$ for a given $v\in V$ (possibly $e^d_v = e^d_w$ for $w\neq v$), see \cite{Lawler}. This way  $\rho(H)=\tau(H^d)$. In particular determining different versions of the category of a finite $T_0$-space turns to be equivalent to determining a transversal number from the viewpoint of graph theory.
\par
These problems are in general $NP$-complete, see \cite{GrayJohnson}. Furthermore, they are related to computing the so called transversal hypergraph and hence to enumerate maximal set of independent sets in a hypergraph \cite{Boros}, or to solve the Boolean function dualization problem
 \cite{Eiter}. These invariants have many applications in Computer Science and in particular to Logic and Artificial Intelligence, as described in \cite{EiterGottlob}. According to \cite{MengFengLi}, the best known algorithm for solving the hypergraph transversal problem is due to Fredman and Khachiyan \cite{Fredman}. Actually, in \cite{MengFengLi} an algorithm for computing the minimum covering (and transversal) number is obtained based on the semi-tensor product of matrices, see \cite{cheng1} and \cite{cheng2}.

 Observe that, by Lemma \ref{lem1} and Remark \ref{crown}, we have a clear criterium about the contractibility of a subspace of a finite $T_0$-space $X$ of height $1$, so it is tempting to recreate the monotonic Boolean function $\sigma(H)$ induced by a Sperner hypergraph $H = (V,E)$ as the monotonic Boolean function associated to the LS-category of a suitable $X$  with $Max(X) = V$.  The following example shows that it is not possible to define such a space.
\begin{example}
\label{Exhyper1}
Consider the hypergraph $H_1=(V_1,E_1)$, where $V_1=\{1,2,3,4,5\}$ and \newline$E_1=\mathcal{P}(V_1)-\{\emptyset, \{1,2,3,4\},\{1,2,3,5\},\{1,2,3,4,5\}\}$. Its Sperner hypergraph has as hyperedges $\hat{E}_1=\{\{1,2,3\}, \{1,2,4,5\}, \{1,3,4,5\}, \{2,3,4,5\}\}$. If there was a finite $T_0$-space $X$ of height $1$ so that $\sigma_{cat}=\sigma(H_1)$, then $\sigma_{cat}$ would be zero on all subsets except for $\{1,2,3,4\}$, $\{1,2,3,5\}$ and $\{1,2,3,4,5\}$. This would mean that $U=U_1\cup U_2\cup U_3\cup U_4$ and $V=U_1\cup U_2\cup U_3\cup U_5$ would contain crowns $X(1)$ and $X(2)$ respectively, as depicted in Figure \ref{dibExhyper1} (see Remark \ref{crown}). Notice that in both of them the order of appearance of the vertices $\{1,2,3\}$ must be same otherwise $\sigma_{cat}(U_2\cup U_3)=1$.

\begin{figure}[ht]
\begin{center}
	\psfrag{v1}{$1$}\psfrag{v2}{$2$}\psfrag{v3}{$3$}\psfrag{v4}{$4$}\psfrag{v5}{$5$}\psfrag{d1}{$d(1)$}\psfrag{d2}{$d(2)$}\psfrag{d3}{$d(3)$}\psfrag{d4}{$d(4)$}\psfrag{d5}{$d(5)$}\psfrag{d1'}{$d(1')$}\psfrag{d2'}{$d(2')$}\psfrag{d3'}{$d(3')$}\psfrag{d4'}{$d(4')$}\psfrag{d5'}{$d(5')$}\psfrag{x1}{$X(1)$}\psfrag{x2}{$X(2)$}
	\psfig{figure=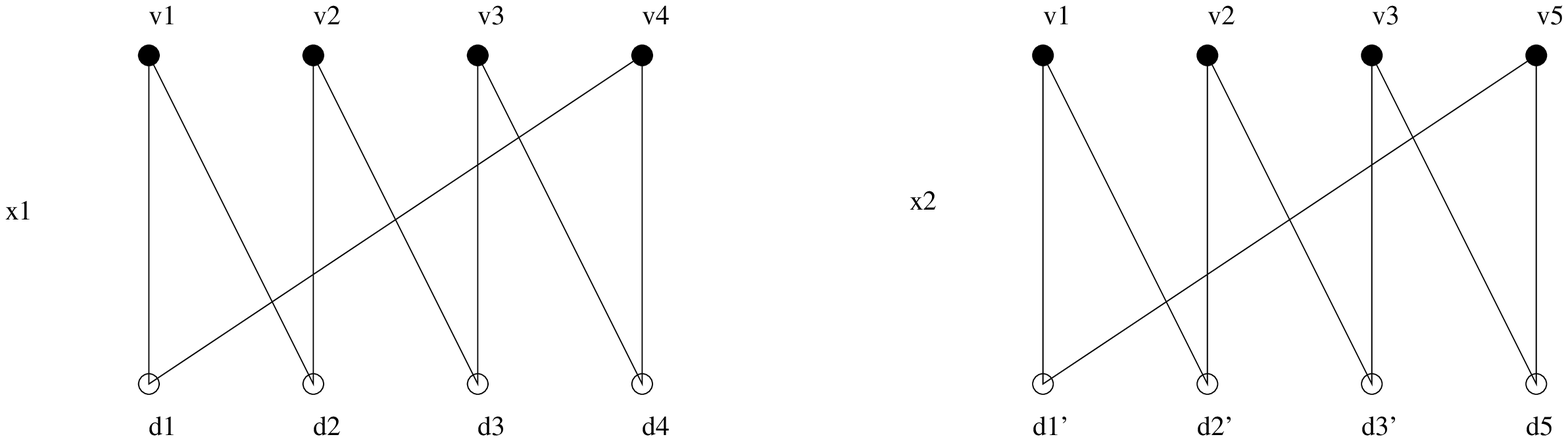,height=3cm}
\end{center}
	\caption{}
\label{dibExhyper1}
\end{figure}

	In these conditions, it should be true that $d(2)=d(2')$ and $d(3)=d(3')$ otherwise $\sigma_{cat}(U_1\cup U_2)=1$ and $\sigma_{cat}(U_2\cup U_3)=1$. Now we study the cases when $d(1)$ and $d(1')$ are or not equal and when $d(4)$ and $d(5)$ are or not equal. In any of the four possible cases we obtain subsets on which $\sigma_{cat}$ takes on the value $1$ when it should not.
\end{example}

It is unknown to the authors whether there exists fixed $k>1$ such the Boolean function $\sigma(H)$ associated to an arbitrary $H$ can be always modeled by the Boolean function associated to the LS-category of a finite space of height $k$. The same conclusion would be obtained regarding the computation of the covering number of a hypergraph and, as apparently every combinatorial problem can be reformulated as the determination of the covering number \cite{Furedi} this would reduce such problems to the calculation of the LS-category of a finite $T_0$-space of height $k$.

In the above recreation of a hypergraph, the collections of subsets of the vertices which are not subsets of any hyperedge play an important role. In the example described above these subsets and the intersections among them have large cardinals with respect to the cardinal of the set of vertices. It seems that if these subsets were small such a recreation would be plausible. The following is an example of such phenomenon.
\begin{example}\label{Exhyper2}
	Let $H_2=(V_2,E_2)$ be the hypergraph with $V_2=\{0,1,2,3,4\}$ and \newline
	$E_2=\mathcal{P}(V_2)-\{\emptyset, U \subseteq V_2 \textrm{ } | \textrm{ } U \mbox{ contains } \{0,1\}, \{1,2\}, \{2,3\},\{3,4\} \mbox{ or } \{0,4\}\}$. Its Sperner hypergraph has as hyperedges $\hat{E}_2=\{\{0,2\},\{0,3\},\{1,3\},\{1,4\}, \{2,4\}\}$. Then the height $1$ finite $T_0$ space $X$ whose Hasse diagram is depicted below has $\sigma_{cat}=\sigma(H_2)$. Its Hasse diagram is the union of the Hasse diagrams of the spaces (crowns)  $X(i,i+1)$ considered as directed graphs where the indices are $i=0,1,\dots, 4\ mod\ 5$.
	\begin{figure}[ht]
\begin{center}
	\psfrag{1}{$1$}\psfrag{2}{$2$}\psfrag{3}{$3$}\psfrag{4}{$4$}\psfrag{0}{$0$}\psfrag{d1}{$d(1)$}\psfrag{d2}{$d(2)$}\psfrag{d3}{$d(3)$}\psfrag{d4}{$d(4)$}\psfrag{d0}{$d(0)$}\psfrag{xi}{$X(i,i+1)\equiv$}\psfrag{x}{$X\equiv$}\psfrag{i}{$i$}\psfrag{i1}{$i+1$}\psfrag{di}{$d(i)$}\psfrag{di1}{$d(i+1)$}
	\psfig{figure=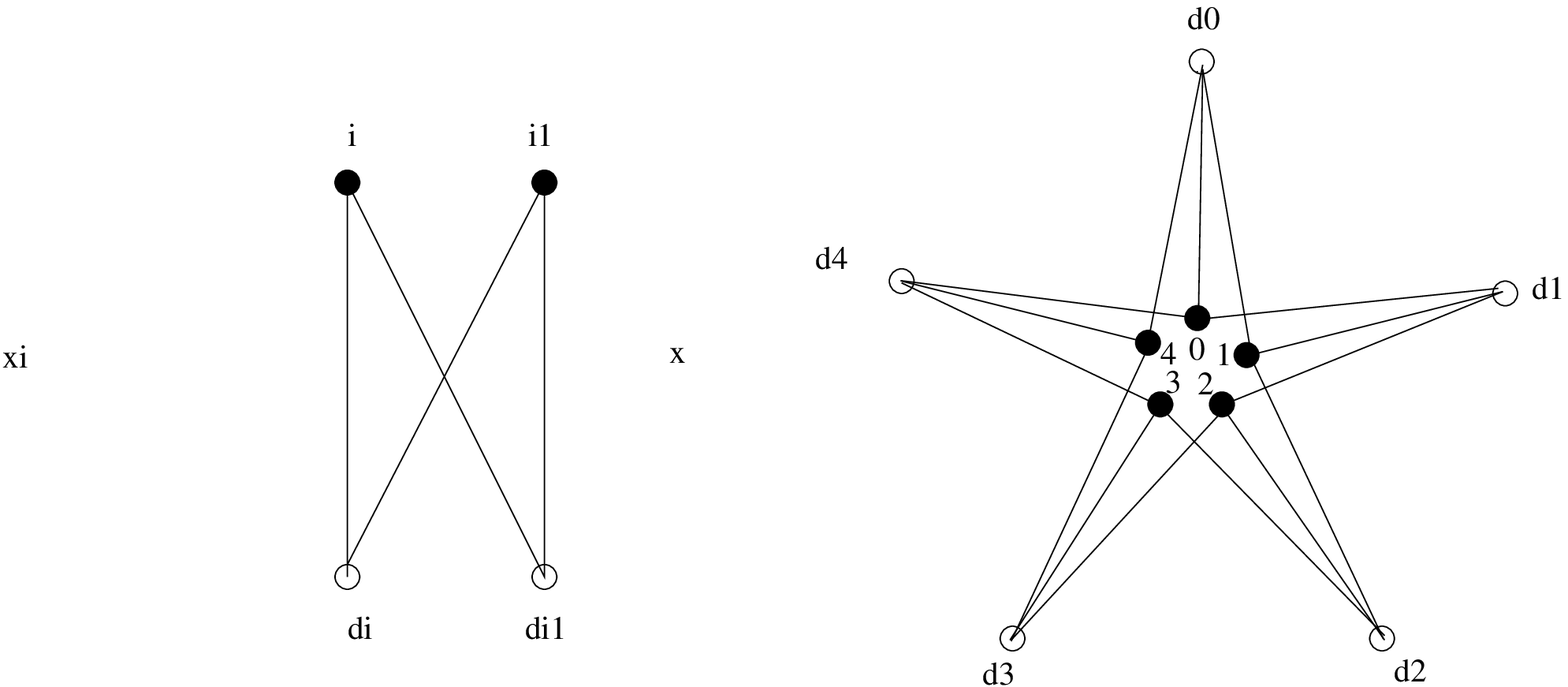,height=5.5cm}
\end{center}
	\caption{}
\label{dibExhyper2}
\end{figure}
\end{example}

We finish with a question that derives from the previous considerations. It seems to be difficult and interesting.

\textbf{Question.} Identify those hypergraphs $H$ for which there exists a finite $T_0$-space of height $1$ with $\sigma_{cat}=\sigma(H)$. That would help to characterize those finite $T_0$-spaces for which the computation of its $LS$-category can be reduced to that of a height $1$ space, where the homotopy issues are under control.

\bibliography{finiteLS}
\bibliographystyle{plain}

\end{document}